\documentclass[12pt]{amsart}

\pdfoutput=1

\usepackage{graphicx}
\usepackage{amsmath,amssymb,amsthm,url}
\usepackage[utf8]{inputenc}
\usepackage[T1]{fontenc}
\usepackage{libertine}
\usepackage[libertine,cmintegrals,cmbraces]{newtxmath}
\usepackage{bbold}
\usepackage[shortlabels]{enumitem}
\usepackage{mathtools}
\usepackage{adjustbox}
\usepackage{microtype}
\usepackage{xcolor}
\usepackage{tikz}
\usepackage{tikz-cd}
\usepackage{nicefrac}
\usepackage{todonotes}
\usepackage{a4wide}
\usepackage{quiver}
\usepackage{anyfontsize}

\AtBeginDocument{%
   \def\MR#1{}
} 
\usepackage{euscript}
\usepackage[all]{xy}

\usepackage[pagebackref]{hyperref}
  
\hypersetup{%
  bookmarksnumbered=true,%%
  colorlinks=true,%
  linkcolor=blue,%
  citecolor=blue,%
  filecolor=blue,%
  menucolor=blue,%
  urlcolor=blue,%
  pdfnewwindow=true,%
  pdfstartview=FitBH}
\usepackage[capitalise]{cleveref}

\usepackage{xcolor}
\definecolor{seagreen}{RGB}{46,139,87}
\definecolor{maroon}{RGB}{128,0,0}
\definecolor{darkviolet}{RGB}{148,0,211}
\definecolor{twelve}{RGB}{100,100,170}
\definecolor{thirteen}{RGB}{100,150,50}
\definecolor{fourteen}{RGB}{200,0,0}
\definecolor{fifteen}{RGB}{0,200,0}
\definecolor{sixteen}{RGB}{0,0,200}
\definecolor{seventeen}{RGB}{200,0,200}
\definecolor{eighteen}{RGB}{0,200,200}

\newcommand{\bb}[1]{\mathbb{#1}}

\newcommand{\es}[1]{\EuScript{#1}}
\renewcommand{\sf}[1]{{\mathsf{#1}}}

\DeclareMathOperator{\colim}{\mathsf{colim}}

\newcommand{\fib}{\mathsf{fib}}

\newcommand{\s}{{\sf{Sp}}}
\newcommand{\T}{\es{S}_\ast}

\newcommand{\poly}[1]{\mathsf{Poly}^{\leq #1}}
\newcommand{\homog}[1]{\mathsf{Homog}^{#1}}

\newcommand{\vect}[1]{\mathsf{Vect}_{#1}}

\newcommand{\Aut}{\mathsf{Aut}}
\newcommand{\Fun}{\sf{Fun}}
\DeclareMathOperator{\Hom}{\mathsf{Hom}}

\DeclareMathOperator{\res}{\mathsf{res}}

\newcommand{\op}{\mathsf{op}}

  \newcommand{\adjunction}[4]{
\xymatrix{
#1:#2 \ar@<.5ex>[r] &
\ar@<.5ex>[l] #3:#4
}}

\newtheorem{thm}{Theorem}[section]

\newtheorem{lem}[thm]{Lemma}

\newtheorem{xxthm}{Theorem}

\theoremstyle{definition}
\newtheorem{definition}[thm]{Definition}

\newtheorem{rem}[thm]{Remark}

\begin{document}
\title{A Dwyer--Rezk classification for polynomial functors in Weiss calculus}

\author{David Barnes}
\address[Barnes]{Queen's University Belfast}
\email{d.barnes@qub.ac.uk}

\author{Magdalena K\k{e}dziorek}
\address[K\k{e}dziorek]{Radboud University Nijmegen}
\email{m.kedziorek@math.ru.nl}

\author{Niall Taggart}
\address[Taggart]{Queen's University Belfast}
\email{n.taggart@qub.ac.uk}

\begin{abstract}
In Goodwillie calculus, unpublished work of Dwyer and Rezk provides a classification of reduced filtered colimit preserving $d$-excisive functors from pointed spaces to spectra as spectrum-valued functors on the category of finite sets of cardinality at most $d$ and epimorphisms. We prove through different methods the analogous result in Weiss calculus: $d$-polynomial functors are equivalent to spectrum-valued functors on the category of finite-dimensional inner product spaces of dimension at most $d$ and orthogonal epimorphisms. Via similar methods we obtain a new proof of the classification of homogeneous functors.
\end{abstract}
\maketitle

\section{Introduction}

Weiss calculus is a homotopy-theoretic tool designed to study geometric problems arising from differential topology with applications throughout geometry and homotopy theory. For $\bb{k}$ any of the skew fields $\bb{R}$, $\bb{C}$ or $\bb{H}$, the calculus approximates a spectrum-valued functor $F: \vect{\bb{k}} \to \s$ from the category of finite-dimensional inner product spaces over $\bb{k}$ and linear isometric embeddings by a tower 
\[\begin{tikzcd}
	&& F \\
	\cdots & {P_dF} & \cdots & {P_1F} & {P_0F}
	\arrow[from=2-4, to=2-5]
	\arrow[from=2-2, to=2-3]
	\arrow[from=2-3, to=2-4]
	\arrow[bend right=30, from=1-3, to=2-2]
	\arrow[from=2-1, to=2-2]
	\arrow[bend left=30, from=1-3, to=2-4]
	\arrow[bend left=20, from=1-3, to=2-5]
\end{tikzcd}\]
of functors in which the functor $P_dF$ is the best approximation to $F$ by a $d$-polynomial functor. 

It is reasonable to think of Weiss calculus as an adaptation of Goodwillie calculus where the role of the finite set of $d$ elements is replaced by the Euclidean space $\bb{k}^d$. This `geometrization' of combinatorics induces subtle complexities within Weiss calculus, and one expects the study of polynomial functors in Weiss calculus to mirror the study of excisive functors from pointed spaces to spectra rather than excisive endofunctors of spectra.

In Goodwillie calculus, there are many known models for the $\infty$-category of $d$-excisive functors between some combination of the categories of (pointed) spaces and spectra. For endofunctors of spectra, Arone--Ching~\cite{AroneChingclassification} provided a model in terms of truncated coalgebras over a comonad, and Glasman~\cite{Glasman} gives a classification in terms of (spectral) Mackey functors on the category $\sf{Epi}_{\leq d}$ of finite sets of cardinality at most $d$ and epimorphisms.   
The difference between these classifications lies in the `classifying objects', i.e., the Arone--Ching classification is in terms of the derivative functors while the Glasman classification is in terms of the cross-effects. 

For functors from pointed spaces to spectra, Arone--Ching~\cite{AroneChingclassification, AroneChingcrosseffects} provide a model again in terms of truncated coalgebras over a comonad using the derivative spectra. However, a simpler model can be given using the cross-effect functors: unpublished work of Dwyer--Rezk describes $d$-excisive functors as spectrum-valued functors on the category $\sf{Epi}_{\leq d}$. 
This contrasts with the Mackey functor description of the spectra to spectra case where one has both `restrictions' and `transfers', in the  pointed spaces to spectra case one just has `restrictions'.
There are several different approaches to models in the style of Dwyer and Rezk, see e.g., work of Arone--Ching~\cite[Proposition 3.15, Theorem 3.82]{AroneChingcrosseffects}, Lurie~\cite[Theorem 6.1.5.1]{HA}, and Chorny--White~\cite[Corollary 5.1]{ChornyWhite}, respectively.

In Weiss calculus, work of Malin and the third author~\cite{MalinTaggart} used the derivatives as the classifying objects to provide a model for spectrum-valued $d$-polynomial functors in terms of truncated coalgebras over a comonad. In this article we prove the following classification analogous to that of Dwyer--Rezk for $d$-excisive functors. 

\begin{xxthm}\label{main thm}
There is an equivalence of $\infty$-categories
\[
\poly{d}(\vect{\bb{k}}, \s) \simeq \Fun(\sf{OEpi}_{\leq d}, \s)
\]
between the $\infty$-category of $d$-polynomial functors and the $\infty$-category of spectrum-valued functors on the category $\sf{OEpi}_{\leq d}$ of finite-dimensional inner product spaces of dimension at most $d$ and orthogonal epimorphisms, see~\cref{def: OEpi}.
\end{xxthm}

Given the analogy between Weiss calculus and Goodwillie calculus our classification uses the induction functors of Weiss, \cite[Proposition 2.1]{Weiss}, which we call \emph{Weiss cross-effects}. These are geometrically oriented versions of the cross-effects from Goodwillie calculus coming from the `geometrization' of replacing the set of $d$ elements with the Euclidean space $\bb{k}^d$.
In a somewhat twisted turn of fate, the inherent geometric aspects of Weiss calculus which are famous for making it a challenging pursuit end up greatly simplifying many of our arguments while proving~\cref{main thm}.

Barwick, Dotto, Glasman, Nardin and Shah~\cite[Example 10]{BDGNS} conjecture that the stable homotopy theory parametrized by the $\infty$-category $\sf{OEpi}_{\leq d}$ ``controls'' Weiss calculus 
and in particular, $d$-polynomial functors are equivalent to the $\infty$-category of spectral Mackey functors on $\sf{OEpi}_{\leq d}$. \cref{main thm} disproves their conjecture: $d$-polynomial functors are not spectral Mackey functors on $\sf{OEpi}_{\leq d}$, or said differently, $d$-polynomial functors have no interesting transfers.

The category $\sf{OEpi}_{\leq d}$ and it's opposite $\vect{\bb{k}, \leq d}$ (see ~\cref{def: OEpi}) are not equivalent as categories, but we provide a Morita equivalence between them by giving another model for $d$-polynomial functors in terms of left Kan extensions along the inclusion $\vect{\bb{k}, \leq d} \hookrightarrow \vect{\bb{k}}$.

\begin{xxthm}\label{main thm: LKE}
There is an adjoint equivalence
\[
\adjunction{\sf{L}_d}{\Fun(\vect{\bb{k},\leq d},\s)}{\poly{d}(\vect{\bb{k}}, \s)}{\res_{\leq d}}
\]
between the $\infty$-category of spectrum-valued functors on the $\infty$-category of finite-dimensional inner product spaces of dimension at most $d$ and the $\infty$-category of $d$-polynomial functors.
\end{xxthm}

This model is the Weiss calculus version of the well known Morita equivalence between the $\infty$-category of finite pointed sets of cardinality at most $d$ and pointed maps, and the $\infty$-category of finite sets of cardinality at most $d$ and epimorphisms, see e.g.,~\cite[Lemma 8.2, Proposition 8.3]{AronePolynomial}. The Goodwillie version of this story has been known for some time, for instance, a  proof that $d$-excisive functors are determined by their restriction to finite pointed sets of cardinality at most $d$ was given by Arone and Ching~\cite[\S3]{AroneChingcrosseffects} and Lurie~\cite[Corollary 6.1.5.2]{HA}. The Morita equivalence between these two categories of finite sets dates back to Pirashvili~\cite{Pirashvili} and various other versions have been proved e.g., by Helmstutler~\cite{Helmstutler} and Walde~\cite{Walde}.

The models of~\cref{main thm} and~\cref{main thm: LKE} have largely mutually exclusive roles to play. The model of~\cref{main thm} provides a model for polynomial functors in terms of truncated $\sf{OEpi}$-comodules, Koszul dual to the model of Malin and the third author~\cite[Theorem D]{MalinTaggart} in terms of structure reminiscent of truncated `divided power' $K(\sf{OEpi})$-modules. On the other hand, the model of~\cref{main thm: LKE} is expected to be useful for answering convergence questions.

The difference between consecutive polynomial approximations 
\[
D_d F = \fib(P_dF \longrightarrow P_{d-1}F),
\]
is a $d$-homogeneous functor, analogous to monomials of degree $d$ from differential calculus. A celebrated theorem of Weiss~\cite[Theorem 7.3]{Weiss} provides an identification of the $\infty$-category of $d$-homogeneous functors as the $\infty$-category of Borel $\Aut(d)$-spectra, i.e., as the $\infty$-category of spectra with an action of the automorphism group $\Aut(d) = \vect{\bb{k}}(\bb{k}^d, \bb{k}^d)$. The proof technique for~\cref{main thm} may be applied to the $\infty$-category of homogeneous functors to provide a new proof of the classification of $d$-homogeneous spectrum-valued functors.

\begin{xxthm}\label{homogeneous classification}
There is an equivalence 
\[
\homog{d}(\vect{\bb{k}}, \s) \simeq \s^{\Aut(d)}
\]
between the $\infty$-category of $d$-homogeneous functors and the $\infty$-category of Borel $\Aut(d)$-spectra, realised by sending a Borel $\Aut(d)$-spectrum $\Theta$ to the functor $\bb{k}^n \mapsto (S^{\bb{k}^d \otimes \bb{k}^n} \wedge \Theta)_{h\Aut(d)}$.
\end{xxthm}

The methods employed here should provide analogous classification results for polynomial and homogeneous functors in other functor calculi. In particular for Goodwillie calculus of functors from pointed spaces to spectra and for the $\sf{FI}$-calculus of Arro~\cite{Arro}. Moreover it is routine to extend to functors taking values in any presentably stable symmetric monoidal $\infty$-category in place of spectra.

\subsection*{Acknowledgements}
We are grateful to Hadrian Heine for useful conversations relating to the enriched $\infty$-categorical aspects of this work. 
The authors would like to thank the Fondation des Treilles for its support and hospitality during the workshop ``Modelling symmetries using algebra'' and the Isaac Newton Institute for Mathematical Sciences, Cambridge, for support and hospitality during the programme \emph{Equivariant homotopy theory in context}, where work on this paper was undertaken. This work was supported by EPSRC grant EP/Z000580/1 and partially supported by a grant from the Simons Foundation. The second and third author were supported by the Nederlandse Organisatie voor Wetenschappelijk Onderzoek (Dutch Research Council) Vidi grant no VI.Vidi.203.004. In the final stages of preparing this article for publication, the third author was supported by the Engineering and Physical Sciences Research Council under grant number EP/Z534705/1.
We thank the anonymous referee for constructive feedback that provided several improvements to the paper.

\section{Polynomial functors}

We assume some familiarity with Weiss calculus, but recap many of the key features throughout. For comprehensive accounts see~\cite{Weiss, TaggartUC, CarrTaggart}, depending on the reader's choice of skew field $\bb{R}$, $\bb{C}$ or $\bb{H}$, respectively. Denote by $\vect{\bb{k}}$ the $\infty$-category of finite-dimensional inner product spaces over $\bb{k}$ and linear isometric embeddings. A concrete model is given by the topological category with space of morphisms $\sf{Vect}_\bb{k}(\bb{k}^m, \bb{k}^n)$ the Stiefel manifold of $m$-frames in $\bb{k}^n$.

\begin{definition}
A functor $F: \vect{\bb{k}} \to \s$ is \emph{$d$-polynomial} if the canonical map 
\[
F(V) \longrightarrow \lim_{0 \neq U \subseteq \bb{k}^{d+1}}F(V\oplus U)
\]
is an equivalence for each $V \in \vect{\bb{k}}$. We define the $\infty$-category $\poly{d}(\vect{\bb{k}}, \s)$ of $d$-polynomial functors to be the full sub-$\infty$-category of $\Fun(\vect{\bb{k}}, \s)$ spanned by the $d$-polynomial functors. 
\end{definition}

Given a functor $F: \vect{\bb{k}} \to \s$, there is a fibre sequence
\begin{equation}\label{eq: poly iff cross effect vanishes}
 F^{(d+1)}(V) \longrightarrow F(V) \longrightarrow \lim_{0 \neq U \subseteq \bb{k}^{d+1}}F(V\oplus U)   
\end{equation}
in which $F^{(d+1)}$ is the $(d+1)$-st Weiss cross-effect, see~\cite[Proposition 5.3 and \S2]{Weiss}. In the stable setting, a functor is $d$-polynomial if and only if the $(d+1)$-st Weiss cross-effect vanishes. We now collect several lemmas
about how cross-effects and polynomial degree interact that will be useful for us later.

\begin{lem}\label{lem: d poly with cross effect zero}
If  $F: \vect{\bb{k}} \to \s$ is a $d$-polynomial functor for which $F^{(d)}(\bb{k}^0) \simeq 0$, then $F$ is $(d-1)$-polynomial.
\end{lem}
\begin{proof}
Weiss cross-effects are related by fibre sequences of the form
\begin{equation}\label{eq: cross-effect fibs}
F^{(n)}(\bb{k}^m) \longrightarrow F^{(n-1)}(\bb{k}^m) \longrightarrow \Omega^{\bb{k}^{n-1}}F^{(n-1)}(\bb{k}^{m+1})
\end{equation}
see e.g.,~\cite[Proposition 2.2]{Weiss}. Since $F$ is $d$-polynomial, the $(d+1)$-st Weiss cross-effect vanishes and so by the $n=d+1$ case of the fibre sequence~\eqref{eq: cross-effect fibs} we have that the map
\[
F^{(d)}(\bb{k}^m) \longrightarrow \Omega^{\bb{k}^{d}}F^{(d)}(\bb{k}^{m+1})
\]
is an equivalence for each $\bb{k}^m$. It follows from the assumption that $F^{(d)}(\bb{k}^0)\simeq 0$ and stability that the $d$-th Weiss cross-effect of $F$ vanishes, and hence by the fibre sequence~\eqref{eq: poly iff cross effect vanishes} that $F$ is $(d-1)$-polynomial. 
\end{proof}

\begin{lem}\label{lem: d poly with cross effects zero}
If  $F: \vect{\bb{k}} \to \s$ is a $d$-polynomial functor for which $F^{(n)}(\bb{k}^0) \simeq 0$ for all $1 \leq n \leq d$, then $F$ is constant.
\end{lem}
\begin{proof}
Applying~\cref{lem: d poly with cross effect zero} $d$-many times we see that $F$ is $0$-polynomial. From~\eqref{eq: poly iff cross effect vanishes} we see that a $0$-polynomial functor is constant.
\end{proof}

Weiss~\cite[Theorem 6.3]{Weiss}\cite{Weisserrata} proved (in different language) that the inclusion 
\[
\sf{inc}_d: \poly{d}(\vect{\bb{k}},\s) \hookrightarrow \Fun(\vect{\bb{k}}, \s)
\]
admits a left exact left adjoint, hence there is an adjunction
\[
\adjunction{P_d}{\Fun(\vect{\bb{k}}, \s)}{\poly{d}(\vect{\bb{k}}, \s)}{\sf{inc}_d}
\]
realising the $\infty$-category of $d$-polynomial functors as a left exact localization of the $\infty$-category of all functors. This localization is compatible (in the sense of~\cite[Definition 2.2.1.6, Example 2.2.1.7]{HA}) with the Day convolution monoidal structure on $\Fun(\vect{\bb{k}}, \s)$. The proof is a spectral lift of~\cite[Theorem 4.2.0.1]{Hendrian}.

\begin{lem}\label{lem: poly d smashing}
The $d$-polynomial approximation $P_d: \Fun(\vect{\bb{k}}, \s) \to \poly{d}(\vect{\bb{k}}, \s)$ is a smashing localization\footnote{A localization is said to be \emph{smashing} if it preserves colimits.} on $\Fun(\vect{\bb{k}}, \s)$ which is compatible with the Day convolution monoidal structure. In particular, the $\infty$-category of $d$-polynomial functors is a presentably symmetric monoidal stable $\infty$-category with monoidal product given by localized Day convolution. 
\end{lem}
\begin{proof}
The only part of the proof not contained in \cite[Theorem 4.2.0.1]{Hendrian} is the observation that $P_d$ is smashing. This follows by the construction~\cite[Theorem 6.3]{Weiss} of $P_d$ as a filtered colimit of compact limits, i.e., $P_d$ preserves filtered colimits and finite limits, and hence all colimits.
\end{proof}

\begin{rem}
It is also the case that the $\infty$-category of $d$-polynomial functors is symmetric monoidal with respect to the localized levelwise monoidal product. This can be seen by leveraging results of Hahn and Yuan~\cite{HahnYuan} on the monoidality of the Weiss tower.
\end{rem}

The equivalence of~\cref{main thm} uses the Schwede--Shipley~\cite{SchwedeShipley} identification of stable $\infty$-categories as $\infty$-categories of modules. To apply this to our setting we first need to observe that the $\infty$-category of $d$-polynomial functors is compactly generated, and to specify a set of compact generators.

\begin{lem}\label{lem: poly compactly generated}
The $\infty$-category $\poly{d}(\vect{\bb{k}}, \s)$ of $d$-polynomial functors is compactly generated. In particular, the set 
\[
G_d =\{\bb{S}^{n(-)}: \bb{k}^m \longmapsto \Sigma^\infty S^{\bb{k}^n \otimes \bb{k}^m} \mid 0 \leq n \leq d\}
\]
is a set of compact generators for $\poly{d}(\vect{\bb{k}}, \s)$.
\end{lem}
\begin{proof}
We begin by showing that the objects of $G_d$ are compact objects of $\poly{d}(\vect{\bb{k}}, \s)$. By~\cite[Proposition 2.1]{Weiss} for each $n \in \bb{N}$, there exist an $\infty$-category $\es{J}_n$ and a (non-full) inclusion $\vect{\bb{k}} \hookrightarrow \es{J}_n$ such that the sphere functor $\bb{S}^{n(-)}$ may be identified with the restriction of the functor $\Sigma^\infty\es{J}_n(\bb{k}^0, -)$ to a functor on $\vect{\bb{k}}$.

Using this identification of $\bb{S}^{n(-)}$ with $\Sigma^\infty\es{J}_n(\bb{k}^0, -)$, the cofibre sequences of~\cite[Proposition 1.1, Proposition 1.2]{Weiss} allow one to inductively construct $\bb{S}^{n(-)}$ from representable functors and hence $\bb{S}^{n(-)}$ is compact in $\Fun(\vect{\bb{k}}, \s)$. By~\cref{lem: poly d smashing}, it follows that $P_d\bb{S}^{n(-)}$ is compact in $\poly{d}(\vect{\bb{k}}, \s)$. The functor $\bb{S}^{n(-)}$ is $n$-polynomial (in fact, $n$-homogeneous by~\cite[Example 5.7]{Weiss}) and so $\bb{S}^{n(-)} \simeq P_d\bb{S}^{n(-)}$, and hence $\bb{S}^{n(-)} $ is a compact object of $\poly{d}(\vect{\bb{k}}, \s)$. 

By~\cite[Lemma 2.2.1]{SchwedeShipley} to show that a set of compact objects generates, it suffices to show that maps out of this set of compact objects detect trivial objects. In our case, we must show that if
\[
\Hom_{\poly{d}(\vect{\bb{k}}, \s)}(\bb{S}^{n(-)}, F) \simeq 0
\]
for all $0 \leq n \leq d$, then $F\simeq 0$. Since by~\cite[Proposition 2.1 and Notation 2.5]{Weiss}
\[
\Hom_{\poly{d}(\vect{\bb{k}}, \s)}(\bb{S}^{n(-)}, F) \simeq \Hom_{\Fun(\vect{\bb{k}}, \s)}(\bb{S}^{n(-)}, F)  \simeq F^{(n)}(\bb{k}^0)
\]
this is equivalent to $F^{(n)}(\bb{k}^0)\simeq 0$ for $0 \leq n \leq d$ implying that $F\simeq 0$. By~\cref{lem: d poly with cross effects zero} we have that if $F^{(n)}(\bb{k}^0)\simeq 0$ for $1 \leq n \leq d$ then $F$ is a constant functor. The additional assumption $F^{(0)}(\bb{k}^0) \simeq 0$ completes the proof since $F(\bb{k}^0) = F^{(0)}(\bb{k}^0)$, and so $F$ is the constant functor at zero.    
\end{proof}

\begin{rem}
The compact generators of~\cref{lem: poly compactly generated} are not (weakly) dualizable with respect to neither the Day convolution nor pointwise monoidal structures. Indeed, dualizability would imply that for each $n$ and every $F: \vect{\bb{k}} \to \s$ the map
\[
\underline{\Hom}(\bb{S}^{n(-)}, \bb{1}) \otimes F \longrightarrow \underline{\Hom}(\bb{S}^{n(-)}, F),
\]
is an equivalence. But in either case the left hand side is always zero since the monoidal unit $\bb{1}$ (for both Day convolution and pointwise monoidal structures) is the constant functor at the sphere which has no non-trivial cross-effects whereas many $F$ have non-trivial cross-effects, i.e., the right-hand side is often non-trivial. For a concrete example take $F$ to be $\bb{S}^{m(-)}$ for $m > n$.  In particular, the $\infty$-category of $d$-polynomial functors is not rigidly-compactly generated.
\end{rem}

\begin{rem}
We will see in~\cref{lem: equiv of sp categories} that the full sub-$\infty$-category spanned by $\{ \bb{S}^{n(-)} \mid n \in \bb{N}\}$ is equivalent to $\sf{OEpi} \simeq \vect{\bb{k}}^\op$. It would be reasonable to conjecture that the full sub-$\infty$-category of the $\infty$-category $\Fun(\vect{\bb{k}}, \s)$ compactly generated by the set $\{ \bb{S}^{n(-)} \mid n \in \bb{N}\}$ is the $\infty$-category of polynomial functors (i.e., those functors which are $d$-polynomial for some $d$). This conjecture is false with a particular counterexample given by the functor $\bigvee_{n < \infty} \bb{S}^{n(-)}$. This functor is an object of the $\infty$-category compactly generated by the set $\{ \bb{S}^{n(-)} \mid n \in \bb{N}\}$ since it is a filtered colimit of polynomial functors, but it itself is not $d$-polynomial for any fixed $d$. We thank the anonymous referee for suggesting this counterexample to us.
\end{rem}

\section{The classification of polynomial functors}
The classification of polynomial functors as in~\cref{main thm} will require us to invoke technology in the form of spectrally enriched $\infty$-categories, e.g., the first step in the classification is an $\infty$-categorical enhancement of the classical Schwede--Shipley theorem~\cite{SchwedeShipley}. We choose to work with Hinich's~\cite{Hinichyoneda, Hinichcolimits} model for enriched $\infty$-categories. We recall the key details here but a comprehensive account may be found in work of Reutter and Zetto~\cite{ReutterZetto}, see also~\cite{Heineequivalence} for a comparison of several of the known models.

We will only define the $\infty$-category of spectrally enriched $\infty$-categories with a specified space of objects since this is all we will need in this work. Let $X$ be a space. The category of $X$-quivers $\Fun(X^\op \times X, \s)$ has a monoidal structure given by
\[
(A \otimes B)(x,z) = \underset{y \in X}{\colim}~ A(x,y) \otimes B(y,z)
\]
see e.g.,~\cite[Corollary 2.29 and Observation 2.32]{ReutterZetto}. Similarly to how enriched (1-)categories may be defined as algebras in the (1-)category of quivers, we have the following definition.

\begin{definition}[{\cite[Definition 3.1.1]{Hinichyoneda}}]
Let $X$ be a space. The $\infty$-category of \emph{spectrally enriched $\infty$-categories with space of objects $X$} is the $\infty$-category of (associative) algebra objects in the category $\Fun(X^\op \times X,\s)$ of $X$-quivers.
\end{definition}

Unpacking this definition amounts to the observation that if $X = \sf{ob}(\es{C})$ for some spectrally enriched $\infty$-category $\es{C}$, the hom object
\[
\Hom_{\es{C}}(-,-) : \sf{ob}(\es{C})^\op \times \sf{ob}(\es{C})  \longrightarrow \s, \quad (C, C') \longmapsto \Hom_{\es{C}}(C,C')
\]
is an $X$-quiver and the coherences in composition give us the (associative) algebra structure.

Let $\es{C}$ be a spectrally enriched $\infty$-category with space of objects $X$. There is an equivalence
\[
\Fun(X^\op \times X, \s) \simeq \sf{End}_\s^{L}(\Fun(X^\op, \s)) = \Fun_\s^L(\Fun(X^\op, \s),\Fun(X^\op, \s))
\]
between the $\infty$-category of $X$-quivers and the $\infty$-category of $\s$-linear colimit preserving endofunctors of $\Fun(X^\op, \s)$, see e.g.,~\cite[Definition 2.30]{ReutterZetto}. Under this equivalence, the monoidal structure above corresponds to composition of functors~\cite[Remark 2.31]{ReutterZetto}. It follows that the category of $X$-quivers acts universally from the left on the category 
\[
\Fun(X^\op, \s)\simeq \Fun(X,\s).
\]
As such we may consider the category of left $\es{C}$-modules in $\Fun(X,\s)$. This leads to a definition of the 
$\infty$-category of spectrally enriched presheaves, see~\cite[Corollary 2.29 and Observation 2.39]{ReutterZetto}.

\begin{definition}\label{def: presheaf}
Let $\es{C}$ be a spectrally enriched $\infty$-category with space of objects $X$. The $\infty$-category of spectrally enriched presheaves $\sf{PSh}_\s(\es{C})$ on $\es{C}$ is the $\infty$-category of left $\es{C}$-modules in $\Fun(X,\s)$.
\end{definition}

We can now begin proving~\cref{main thm} using the $\infty$-categorical enhancement~\cite[Theorem 4.1]{AroneBarneaSchlank} of the Schwede--Shipley theorem~\cite[Theorem 3.3.3]{SchwedeShipley}. For this, let $\es{G}_d$ be the full spectrally enriched sub-$\infty$-category of $\Fun(\vect{\bb{k}}, \s)$ spanned by the set $G_d$ of compact generators for $\poly{d}(\vect{\bb{k}}, \s)$ from~\cref{lem: poly compactly generated}.

\begin{lem}\label{lem: SS for poly}
There is an equivalence
\[
\poly{d}({\vect{\bb{k}}}, \s) \xrightarrow{ \ \simeq \ } \sf{PSh}_\s(\es{G}_d)
\]
between the $\infty$-category of $d$-polynomial functors and the $\infty$-category of spectral presheaves on $\es{G}_d$.
\end{lem}
\begin{proof}
Combine~\cref{lem: poly compactly generated} with the $\infty$-categorical Schwede--Shipley Theorem of~\cite[Theorem 4.1]{AroneBarneaSchlank}.
\end{proof}

The suspension spectrum functor $\Sigma^\infty_+ : \es{S} \to \s$ induces a left adjoint functor
\[
\Sigma^\infty_+: \sf{Cat}_\infty \longrightarrow \sf{Cat}_\infty^\s
\]
relating $\infty$-categories and spectrally enriched $\infty$-categories, see e.g.,~\cite[Construction 2.36]{ReutterZetto}, such that the space of objects is preserved and on enriched mapping objects
\[
\Hom_{\Sigma^\infty_+\es{C}}(X,Y) \simeq \Sigma^\infty_+ \Hom_\es{C}(X,Y).
\]
We call the image $\Sigma^\infty_+\es{C}$, the free spectrally enriched $\infty$-category generated by $\es{C}$. In~\cref{lem: SS for poly} we reduced the proof of~\cref{main thm} to proving that the spectral $\infty$-category $\es{G}_d$ is equivalent to the free spectrally enriched $\infty$-category generated by $\sf{OEpi}_{\leq d}$, defined below.

\begin{definition}\label{def: OEpi}
Let $\vect{\bb{k}}$ denote the $\infty$-category of finite-dimensional inner product spaces and linear isometric embeddings. The space of morphisms $\vect{\bb{k}}(\bb{k}^m, \bb{k}^n)$ is the Stiefel manifold of $m$-frames in $\bb{k}^n$. We define the $\infty$-category $\sf{OEpi}$ of finite-dimensional inner product spaces and orthogonal epimorphisms to be $\vect{\bb{k}}^\op$. We will denote by $\vect{\bb{k}, \leq d}$ and $\sf{OEpi}_{\leq d}$ the full sub-$\infty$-categories spanned by those inner product spaces with dimension at most $d$. 
\end{definition}

\begin{lem}\label{lem: equiv of sp categories}
There is an equivalence
\[
\Sigma^\infty_+ \sf{OEpi}_{\leq d} \longrightarrow (\es{G}_d)^\op
\]
of spectrally enriched $\infty$-categories.
\end{lem}
\begin{proof}
The functor is easier to describe on opposite categories. On opposite categories it is given by
\[
\bb{S}^{\ast(-)} : \Sigma^\infty_+ \vect{\bb{k}, {\leq d}} \longrightarrow \es{G}_d, \quad \bb{k}^n \longmapsto \bb{S}^{n(-)}
\]
on objects and sends an isometric embedding $\bb{k}^m \to \bb{k}^n$ to the induced natural transformation
\[
\bb{S}^{m(-)} = \bb{S}^{\bb{k}^m \otimes -} \longrightarrow \bb{S}^{\bb{k}^n \otimes -} = \bb{S}^{n(-)}.
\]
To see that this assignment defines a spectrally enriched functor, observe that in this case a spectrally enriched functor is by \cref{def: presheaf} (see also, \cite[6.1.2, 6.1.3]{Hinichyoneda}) a $\Sigma^\infty_+ \vect{\bb{k}}$-module in the $\infty$-category of functors $\Fun(\Sigma^\infty_+ \vect{\bb{k}},\Fun(\vect{\bb{k}}, \s))$. The assignment above provides the underlying object of the module, and the associative and unital module structure maps are given by 
\[
\Sigma^\infty (\vect{\bb{k}}(\bb{k}^m, \bb{k}^n)_+ \wedge S^{\bb{k}^m \otimes (-)} \longrightarrow S^{\bb{k}^n \otimes (-)})
\]
which are continuous since the one-point compactification functor is a well-defined functor $\vect{\bb{k}} \to \es{S}_\ast$.  We will abuse notation and denote the induced functor on opposite categories by 
\[
\bb{S}^{\ast(-)} : \Sigma^\infty_+ \sf{OEpi}_{\leq d} \longrightarrow (\es{G}_d)^\op.
\]

To see that the functor $\bb{S}^{\ast(-)} $ is an equivalence of spectrally enriched $\infty$-categories it suffices to show that it is essentially surjective and fully faithful. Essential surjectivity follows from the observation that the functor $\bb{S}^{\ast(-)}$ induces a bijection on the space of objects: both spaces of objects are equivalent to the set of $d+1$ elements.

For fully faithful, we must show that the induced map on mapping spectra
\[
\Sigma^\infty_+\vect{\bb{k}}(\bb{k}^m, \bb{k}^n) \longrightarrow \Hom_{\Fun(\vect{\bb{k}}, \s)}(\bb{S}^{m(-)}, \bb{S}^{n(-)})
\]
is an equivalence. By definition, the right-hand side of this equivalence may be identified with $(\bb{S}^{n(-)})^{(m)}(\bb{k}^0)$, the $m$-th Weiss cross-effect of $\bb{S}^{n(-)}$ evaluated at $\bb{k}^0$. This cross-effect may be computed as in ~\cite[Example 5.7]{Weiss} as
\[
(\bb{S}^{n(-)})^{(m)}(\bb{k}^0) \simeq (\Sigma^\infty_+ \Aut(n) \wedge S^{\bb{k}^n \otimes \bb{k}^0})_{h\Aut(n-m)} \simeq \Aut(n)/\Aut(n-m) \cong \vect{\bb{k}}(\bb{k}^m, \bb{k}^n).
\]
The statement that the induced map on mapping spectra is an equivalence follows from carefully tracing through the computation of ~\cite[Example 5.7]{Weiss}.
\end{proof}

Fully-faithfulness of the above functor may also be deduced from~\cite[Theorem 3.2]{AroneDwyerLesh} which gives a formula for computing natural transformations between homogeneous functors. By examining the proof of loc. cit. in the special case that the homogeneous functors are of the form $\bb{S}^{n(-)}$ one readily sees that loc. cit. reduces to \cite[Example 5.7]{Weiss}.

The final step in the proof of~\cref{main thm} will be showing that for an $\infty$-category $\es{C}$, the spectrally enriched $\infty$-category $\Sigma^\infty_+\es{C}$ really is the free spectrally enriched $\infty$-category generated by $\es{C}$. This is another application of the Schwede-Shipley theorem.

\begin{lem}\label{lem: top vs. sp enriched}
Let $\es{C}$ be an $\infty$-category. There is an equivalence of $\infty$-categories
\[
\Fun_\s(\Sigma^\infty_+\es{C}, \s) \simeq \Fun(\es{C}, \s)
\]
between the $\infty$-category of spectrally enriched functors on $\Sigma^\infty_+\es{C}$ and the $\infty$-category of spectrum-valued functors on $\es{C}$. 
\end{lem}
\begin{proof}
The $\infty$-category $\Fun(\es{C}, \s)$ is compactly generated by the representable functors and hence the $\infty$-categorical Schwede-Shipley Theorem~\cite[Theorem 4.1]{AroneBarneaSchlank} implies that there is an equivalence
\[
\Fun(\es{C}, \s) \simeq \sf{PSh}_{\s}(\langle \Sigma^\infty_+ \es{C}(C,-) \mid C \in \es{C}\rangle)
\]
between the $\infty$-category $\Fun(\es{C}, \s)$ and the $\infty$-category of spectrally enriched presheaves on the full spectrally enriched sub-$\infty$-category $\langle \Sigma^\infty_+ \es{C}(C,-) \mid C \in \es{C}\rangle$ of $\Fun(\es{C}, \s)$ spanned by the representable functors. The enriched Yoneda embedding (see e.g.,~\cite[Theorem 5.5]{ReutterZetto} and~\cite[6.2.7]{Hinichyoneda}) provides a functor
\[
\Sigma^\infty_+ \es{C}^\op \longrightarrow \langle \Sigma^\infty_+ \es{C}(C,-) \mid C \in \es{C}\rangle
\]
which is an equivalence of spectrally enriched $\infty$-categories. Fully faithfulness follows from the fully faithfulness of the enriched Yoneda embedding and essential surjectivity follows since both spaces of objects may be identified with the objects of $\es{C}$. Taking spectral presheaves provides the desired equivalence 
\[
\Fun_\s(\Sigma^\infty_+\es{C}, \s) = \sf{PSh}_\s(\Sigma^\infty_+ \es{C}^\op) \simeq \sf{PSh}_\s(\langle\Sigma^\infty_+ \es{C}(C,-) \mid C \in \es{C}\rangle) \simeq \Fun(\es{C}, \s).   \qedhere
\]\qedhere
\end{proof}

We now prove~\cref{main thm} by essentially combining the previous three lemmas.

\begin{thm}[\cref{main thm}]
There is an equivalence of $\infty$-categories
\[
\poly{d}(\vect{\bb{k}}, \s) \simeq \Fun(\sf{OEpi}_{\leq d}, \s)
\]
between the $\infty$-category of $d$-polynomial functors and the $\infty$-category of spectrum-valued functors on $\sf{OEpi}_{\leq d}$.
\end{thm}
\begin{proof}
By~\cref{lem: SS for poly} we may identify the $\infty$-category of $d$-polynomial functors with the $\infty$-category of spectral presheaves of $\es{G}_d$, where $\es{G}_d$ is the enriched $\infty$-category generated by the set of compact generators for the $\infty$-category of $d$-polynomial functors. By~\cref{lem: equiv of sp categories}, there is an equivalence of spectrally enriched $\infty$-categories
\[
\Sigma^\infty_+ \sf{OEpi}_{\leq d} \longrightarrow (\es{G}_d)^\op
\]
and an application of~\cite[Lemma 6.45]{ReutterZetto} asserts that an equivalence of spectrally enriched categories induces an equivalence of enriched functor categories. It follows that the $\infty$-category of $d$-polynomial functors is equivalent to spectrally enriched functors on $\Sigma^\infty_+ \sf{OEpi}_{\leq d}$. The proof is complete upon an application of~\cref{lem: top vs. sp enriched}.
\end{proof}

We conclude this section by providing another model for the $\infty$-category of $d$-polynomial functors in terms of vector spaces of dimension at most $d$ by showing that a $d$-polynomial functor is equivalent to the left Kan extension of its restriction to $\vect{\bb{k}, \leq d}$. This provides a Morita equivalence between $\vect{\bb{k}, \leq d}$ and it's opposite $\sf{OEpi}_{\leq d}$ through~\cref{main thm}. 

\begin{thm}[\cref{main thm: LKE}]\label{thm: LKE model}
There is an adjoint equivalence
\[
\adjunction{\sf{L}_d}{\Fun(\vect{\bb{k},\leq d},\s)}{\poly{d}(\vect{\bb{k}}, \s)}{\res_{\leq d}}
\]
between the $\infty$-category of spectrum-valued functors on the $\infty$-category of finite-dimensional inner product spaces of dimension at most $d$ and the $\infty$-category of $d$-polynomial functors.
\end{thm}
\begin{proof}
The fully faithful inclusion of $\infty$-categories $\vect{\bb{k},\leq d} \hookrightarrow \vect{\bb{k}}$ induces an adjoint pair
\[
\adjunction{\sf{L}_d}{\Fun(\vect{\bb{k},\leq d},\s)}{\Fun(\vect{\bb{k}}, \s)}{\res_{\leq d}}
\]
given by left Kan extension and restriction. Since the inclusion is fully faithful, the unit of the adjunction is an equivalence. 

The $\infty$-category of $d$-polynomial functors is closed under colimits since it is closed under filtered colimits and finite colimits, the latter by stability. Let $F$ be a functor in the essential image of the left Kan extension. By the coend formula for left Kan extensions it follows that $F$ is a colimit of functors of the form $X \wedge \Sigma^\infty_+\vect{\bb{k}}(\bb{k}^m, -)$ for $0 \leq m \leq d$ and $X \in \s$. The (co)fibre sequences of~\cite[Proposition 1.1, Proposition 1.2]{Weiss} together with~\cite[Example 5.8]{Weiss} imply that the representable functor $\Sigma^\infty_+\vect{\bb{k}}(\bb{k}^m, -)$ is $m$-polynomial. It follows that functors of the form $X \wedge \Sigma^\infty_+\vect{\bb{k}}(\bb{k}^m, -)$ for $0 \leq m \leq d$ are $d$-polynomial and hence the left Kan extension factors through the $\infty$-category of $d$-polynomial functors. This produces an adjunction
\[
\adjunction{\sf{L}_d}{\Fun(\vect{\bb{k},\leq d},\s)}{\poly{d}(\vect{\bb{k}}, \s)}{\res_{\leq d}}
\]
for which the unit remains an equivalence since the $\infty$-category of $d$-polynomial functors is a full subcategory of the $\infty$-category of all functors. In particular, the left Kan extension is fully faithful. 

To see that the left Kan extension induces an equivalence, it suffices to show that it is essentially surjective. For this, it suffices by~\cref{lem: poly compactly generated} and the fact that the left Kan extension commutes with colimits to show that the functors $\bb{S}^{n(-)}$ for $0 \leq n \leq d$ are in the essential image of the left Kan extension functor. This again follows from the cofibre sequences~\cite[Proposition 1.1, Proposition 1.2]{Weiss} since tracking through the construction of $\bb{S}^{n(-)}$ as a finite colimit of representable functors shows that it only requires the representable functors $\Sigma^\infty_+\vect{\bb{k}}(\bb{k}^m, -)$ for $0 \leq m \leq n$.
\end{proof}

\begin{rem}
The models of \cref{main thm} and \cref{main thm: LKE} are not compatible with the Weiss tower. For instance, with respect to~\cref{main thm}, the diagrams
\[\begin{tikzcd}[column sep=1em]
	{\poly{d}_\ast(\vect{\bb{k}}, \s)} & {\Fun(\sf{OEpi}_{\leq d},\s)} & {\poly{d}_\ast(\vect{\bb{k}}, \s)} & {\Fun(\sf{OEpi}_{\leq d},\s)} \\
	{\poly{d-1}_\ast(\vect{\bb{k}}, \s)} & {\Fun(\sf{OEpi}_{\leq (d-1)},\s)} & {\poly{d-1}_\ast(\vect{\bb{k}}, \s)} & {\Fun(\sf{OEpi}_{\leq (d-1)},\s)}
	\arrow["\simeq", tail reversed, no head, from=1-1, to=1-2]
	\arrow["{P_{d-1}}"', from=1-1, to=2-1]
	\arrow["{i^\ast}", from=1-2, to=2-2]
	\arrow["\simeq", from=1-3, to=1-4]
	\arrow["{P_{d-1}}"', from=1-3, to=2-3]
	\arrow["{i^\ast}", from=1-4, to=2-4]
	\arrow["\simeq", tail reversed, no head, from=2-1, to=2-2]
	\arrow["\simeq"', from=2-3, to=2-4]
\end{tikzcd}\]
in which $i : \sf{OEpi}_{\leq (d-1)} \hookrightarrow \sf{OEpi}_{\leq d}$ is the fully faithful inclusion \emph{do not} commute. For the left-most diagram, this follows by considering the diagram of right adjoints which can be seen to not commute by computing the right Kan extension along $i$. Commutativity of the right-most diagram would imply that $P_{d-1}F(\bb{k}^0) \simeq F(\bb{k}^0)$, which is not the case, e.g., for $F = \bb{S}^{d(-)}$, an equivalence $P_{d-1}F(\bb{k}^0) \simeq F(\bb{k}^0)$ would assert that the sphere spectrum is trivial or equivalently that the stable homotopy groups of spheres are zero. For the model of~\cref{main thm: LKE}, the diagrams
\[\begin{tikzcd}[column sep=1em]
	{\poly{d}_\ast(\vect{\bb{k}}, \s)} & {\Fun(\vect{\bb{k}, \leq d},\s)} & {\poly{d}_\ast(\vect{\bb{k}}, \s)} & {\Fun(\vect{\bb{k}, \leq d},\s)} \\
	{\poly{d-1}_\ast(\vect{\bb{k}}, \s)} & {\Fun(\vect{\bb{k}, \leq (d-1)},\s)} & {\poly{d-1}_\ast(\vect{\bb{k}}, \s)} & {\Fun(\vect{\bb{k}, \leq (d-1)},\s)}
	\arrow["\simeq", tail reversed, no head, from=1-1, to=1-2]
	\arrow["{P_{d-1}}"', from=1-1, to=2-1]
	\arrow["{i^\ast}", from=1-2, to=2-2]
	\arrow["\simeq", from=1-3, to=1-4]
	\arrow["{P_{d-1}}"', from=1-3, to=2-3]
	\arrow["{i^\ast}", from=1-4, to=2-4]
	\arrow["\simeq", tail reversed, no head, from=2-1, to=2-2]
	\arrow["\simeq"', from=2-3, to=2-4]
\end{tikzcd}\]
\emph{do not} commute for analogous reasons. There are of course many more choices of what one could mean by ``compatible'' but those above appear to be the only choices which would have applications, e.g., to the classification of homogeneous functors. 
\end{rem}

\section{Homogeneous functors}
A functor $F: \vect{\bb{k}} \to \s$ is $d$-homogeneous if it is $d$-polynomial and $P_{d-1}F \simeq 0$. Define the $\infty$-category $\homog{d}(\vect{\bb{k}}, \s)$ of $d$-homogeneous functors to be the full sub-$\infty$-category of $\poly{d}(\vect{\bb{k}}, \s)$ spanned by the $d$-homogeneous functors. The primary example of a $d$-homogeneous functor is the $d$-th layer
\[
D_dF = \fib(P_dF \longrightarrow P_{d-1}F)
\]
of the Weiss tower. In this section we  
provide a new proof of the classification of $d$-homogeneous functors as Borel $\Aut(d)$-spectra. 

\begin{lem}
The $\infty$-category $\homog{d}(\vect{\bb{k}}, \s)$ of $d$-homogeneous functors is compactly generated by the functor $\bb{S}^{d(-)}$.
\end{lem}
\begin{proof}
We begin by showing that the functor $\bb{S}^{d(-)}$ is compact in $\homog{d}(\vect{\bb{k}}, \s)$. For this, observe that colimits (and in particular, coproducts) in the $\infty$-category of $d$-homogeneous functors are computed by applying the coreflection functor $D_d$ to the colimit in the $\infty$-category of $d$-polynomial functors. Since the coreflection functor $D_d$ is right adjoint to the inclusion, there is an equivalence 
\[
\Hom_{\homog{d}(\vect{\bb{k}}, \s)}\left(\bb{S}^{d(-)}, D_d\coprod_{i \in I} F_i\right) \simeq \Hom_{\poly{d}(\vect{\bb{k}}, \s)}\left(\bb{S}^{d(-)}, \coprod_{i \in I} F_i\right).
\]
By \cref{lem: poly compactly generated} the functor $\bb{S}^{d(-)}$ is compact in the $\infty$-category of $d$-polynomial functors and so there is a further equivalence, 
\[
\Hom_{\poly{d}(\vect{\bb{k}}, \s)}\left(\bb{S}^{d(-)}, \coprod_{i \in I} F_i\right) \simeq \prod_{i \in I}\Hom_{\poly{d}(\vect{\bb{k}}, \s)}(\bb{S}^{d(-)}, F_i),
\]
which completes the proof that $\bb{S}^{d(-)}$ is compact in $\homog{d}(\vect{\bb{k}}, \s)$ upon the observation that for $d$-homogeneous functors $F$ and $G$, 
\[
\Hom_{\homog{d}(\vect{\bb{k}}, \s)}(F, G) \simeq \Hom_{\poly{d}(\vect{\bb{k}}, \s)}(F, G)
\]
by the fact that the $\infty$-category of $d$-homogeneous functors is a full subcategory of the $\infty$-category of $d$-polynomial functors. 

To see that $\bb{S}^{d(-)}$ generates the $\infty$-category of $d$-homogeneous functors it suffices (see the proof of~\cref{lem: poly compactly generated}) to show that if 
\[
\Hom_{\homog{d}(\vect{\bb{k}}, \s)}(\bb{S}^{d(-)}, F) \simeq 0,
\]
then $F \simeq 0$. Since $\homog{d}(\vect{\bb{k}}, \s)$ is a full subcategory of $\Fun(\vect{\bb{k}}, \s)$, this is equivalent to the statement that if $F^{(d)}(\bb{k}^0) \simeq 0$, then $F \simeq 0$ for $F$ a $d$-homogeneous functor. By~\cref{lem: d poly with cross effect zero}, if $F^{(d)}(\bb{k}^0) \simeq 0$, then $F$ is $(d-1)$-polynomial. However, $(d-1)$-polynomial functors are trivial in $\homog{d}(\vect{\bb{k}}, \s)$, hence $F$ is trivial.
\end{proof}

Mirroring the proof of~\cref{main thm} provides a new proof of the classification of homogeneous functors. This proof is much more straightforward than one obtained by transporting the proof of~\cite[Theorem 7.3]{Weiss} into the stable setting, see e.g.,~\cite[Theorem 11.3]{BarnesOman}.

\begin{thm}[\cref{homogeneous classification}]
There is an equivalence 
\[
\homog{d}(\vect{\bb{k}}, \s) \simeq \s^{B\Aut(d)}
\]
realised by sending a Borel $\Aut(d)$-spectrum $\Theta$ to the functor $\bb{k}^n \mapsto (S^{\bb{k}^d \otimes \bb{k}^n} \wedge \Theta)_{h\Aut(d)}$.
\end{thm}
\begin{proof}
Let $\langle \bb{S}^{d(-)}\rangle$ be the spectrally enriched sub-$\infty$-category of $\homog{d}(\vect{\bb{k}}, \s)$ on the functor $\bb{S}^{d(-)}$. There is a sequence of equivalences
\[
\homog{d}(\vect{\bb{k}}, \s) \simeq \sf{PSh}_\s(\langle \bb{S}^{d(-)}\rangle) \simeq \sf{PSh}_\s(\Sigma^\infty_+ B\Aut(d)) \simeq \s^{B\Aut(d)}
\]
by first applying the Schwede-Shipley equivalence, then the equivalence of~\cref{lem: equiv of sp categories}, and finally~\cref{lem: top vs. sp enriched}.

It is left to identify the functor realising the equivalence. Through this equivalence, a spectrum $\Theta$ with an action of $\Aut(d)$ is equivalent to a presheaf $M_\Theta$ on $\langle \bb{S}^{d(-)}\rangle$, i.e., $M_\Theta(\bb{S}^{d(-)}) = \Theta$. The left adjoint to Schwede-Shipley equivalence is given by sending a presheaf $M$ on $\langle \bb{S}^{d(-)}\rangle$ to the enriched coend of the bifunctor
\[
\langle \bb{S}^{d(-)}\rangle\times \langle \bb{S}^{d(-)}\rangle^\op \longrightarrow \s, \quad (\bb{S}^{d(-)}, \bb{S}^{d(-)}) \longmapsto  \bb{S}^{d(-)} \wedge M(\bb{S}^{d(-)}).
\]
Since $M_\Theta$ is supported only in dimension $d$, the image of $M_\Theta$ under this enriched coend is given by
\[
\underset{{B\Aut(d)}}{\colim}~(\bb{S}^{d(-)} \wedge \Theta).
\]
As such, the image of $\Theta$ under the sequence of equivalences is the functor $(\bb{S}^{d(-)} \wedge \Theta)_{h\Aut(d)}$. 
\end{proof}

\begin{rem}
It is not clear how to lift~\cref{homogeneous classification} to a statement about space-valued homogeneous functors. The ideal method (mirroring Goodwillie calculus) would be to show that the functor 
\[
\Omega^\infty: \homog{d}(\vect{\bb{k}}, \s) \longrightarrow \homog{d}(\vect{\bb{k}}, \T),
\]
induced by postcomposition with $\Omega^\infty: \s \to \T$ is an equivalence. This would follow from the existence of a delooping functor $B: \homog{d}(\vect{\bb{k}}, \T) \to \homog{d}(\vect{\bb{k}}, \T)$. It is straightforward to prove that $B = P_d(\Sigma(-))$ provides a left adjoint to the loops functor $\Omega$ on $\homog{d}(\vect{\bb{k}}, \T)$, but the proof~\cite[Lemma 7.4]{Weiss} that this adjunction is an equivalence crucially uses the classification of space-valued homogeneous functors~\cite[Theorem 7.3]{Weiss}. 
\end{rem}

\bibliography{references}
\bibliographystyle{alpha}
\end{document}